\documentclass[a4paper,12pt]{article}
\usepackage{a4wide}
\usepackage{amsmath}
\usepackage{amssymb}
\usepackage{amsthm}
\usepackage{latexsym}
\usepackage{graphicx}
\usepackage[english]{babel}
\usepackage{makeidx}
\usepackage{mathtools}

\newtheorem{obs} [subsection]{Remark}

\newtheorem{prop}[subsection]{Proposition}

\newtheorem{teor}[subsection]{Theorem}
\newtheorem*{teor*}{Theorem}

\newtheorem{cor} [subsection]{Corollary}

\newcommand{\pa}{p_{\mathbf a}}

\newcommand{\Pa}{P_{\mathbf a}}

\newcommand{\Res}{Res}

\newcommand{\stir}{\genfrac{[}{]}{0pt}{}}

\DeclarePairedDelimiter\ceil{\lceil}{\rceil}
\DeclarePairedDelimiter\floor{\lfloor}{\rfloor}

\def\lcm{\operatorname{lcm}}
\def\gcd{\operatorname{gcd}}

\begin{document}
\selectlanguage{english}
% \frenchspacing
\numberwithin{equation}{section}

\large
\begin{center}
\textbf{A note on the number of partitions of $n$ into $k$ parts}

Mircea Cimpoea\c s
\end{center}
\normalsize

\begin{abstract}
We prove new formulas and congruences for $p(n,k):=$ the number of partitions of $n$ into $k$ 
parts and $q(n,k):=$ the number of partitions of $n$ into $k$ distinct parts. Also, we give
lower and upper bounds for the density of the set $\{n\in\mathbb N\;:\;p(n,k)\equiv i(\bmod\; m)\}$, 
where $m\geq 2$ and $0\leq i\leq m-1$.

\noindent \textbf{Keywords:} Restricted integer partitions; Restricted partition function.

\noindent \textbf{2010 MSC:} Primary 11P81 ; Secondary 11P83
\end{abstract}

\section{Introduction}

A partition of a positive integer $n$ is a non-increasing sequence of positive integers whose sum equals $n$. We define
$p(n)$ as the number of partitions of $n$ and for convenience, we define $p(0) = 1$.
Let $p(n,k)$ be the number of partitions of $n$ with exactly $k$ summands. 
Let $q(n,k)$ be the number of partitions of $n$ with $k$ distinct parts 
and let $q(n)$ total number of partitions of $n$ with distinct parts. 
For instance, there are $5$ partitions of $8$ with three summands
$1+1+6,\;1+2+5,\;1+3+4,\;2+2+4,\;2+3+3$,
hence $p(8,3)=5$ and $q(8,3)=2$.
%It is well known that
%\begin{equation}\label{recu}
%p(n,k)=p(n-1,k-1)+p(n-k,k),
%\end{equation} for all $n,k\geq 1$. 
Obviously, $p(n,k)=0$ if and only if $n<k$. Also, $q(n,k)=0$ if and only if $n<k+\binom{k}{2}$.
Moreover, $p(n)=\sum_{k=1}^n p(n,k)\text{ and }q(n)=\sum_{k=1}^n q(n,k).$
The function $p(n,k)$ was studied extensively in the literature; see for instance \cite{intrator}.
However, there is no known closed form for $p(n,k)$. 

Let $\mathbf a=(a_1,\ldots,a_k)$ be a sequence of positive integers
and let $\pa(n)$ be the restricted partition function associated to $\mathbf a$; see Section $2$.
In Theorem \ref{formu}, we prove new formulas for 
$p(n,k)$ and $q(n,k)$, using their intrinsic connection with the restricted partition function associated to 
the sequence $\mathbf k:=(1,2,\ldots,k)$. In Proposition \ref{trei}, we give another formulas for $p(n,3)$ and $q(n,3)$.
%$p_{(1,2,\ldots,k)}(n-k)$.
In \cite{medi} we proved that if a certain determinant is nonzero, then the restricted partition function $\pa(n)$ can
be computed by solving a system of linear equations with coefficients which are values of Bernoulli
polynomials and Bernoulli Barnes numbers. Using a similar method, we prove that if a certain determinant $\Delta(k)$, which
depends only on $k$, is nonzero, then $p(n,k)$ and $q(n,k)$ can be expressed in terms of values of Bernoulli
polynomials and Bernoulli Barnes numbers; see Theorem \ref{dete}.

In Theorem \ref{poli}, respectively in Corollary \ref{cori}, we provide formulas for $P(n,k)=$ the polynomial part of $p(n,k)$,
respectively for $Q(n,k)=$ the polynomial part of $q(n,k)$. 
In Proposition \ref{unde} we prove formulas for the "waves" of $p(n,k)$ and
$q(n,k)$, defined analogously as the Sylvester "waves" (see \cite{sylvester},\cite{sylv}) of the restricted partition function
$p_{\mathbf k}(n)$.

In Proposition \ref{p61} we give new formulas for $p(n,k)$ and $q(n,k)$ in terms of coefficients of a reciprocal polynomial and,
as a consequence, in Corollary \ref{c62}, we prove some congruence relations for $p(n,k)$ and $q(n,k)$.
In a recent preprint \cite{graj}, K. Grajdzica found lower and upper bounds for the density of the set 
$\{n\in\mathbb N\;:\;\pa(n)\equiv i(\bmod\; m)\}$ for a fixed integer $0\leq i\leq m-1$. Using this, we prove 
lower and upper bounds for the density of the set  $\{n\in\mathbb N\;:\;p(n,k)\equiv i(\bmod\; m)\}$; see Theorem \ref{low}.

\newpage
\section{Preliminaries}

Let $\mathbf a := (a_1, a_2, \ldots , a_k)$ be a sequence of positive integers, $k \geq 1$. The \emph{restricted partition
function} associated to $\mathbf a$ is $\pa : \mathbb N \to \mathbb N$, $\pa(n) :=$ the number of integer solutions $(x_1, \ldots, x_k)$
of $\sum_{i=1}^k a_ix_i = n$ with $x_i \geq 0$. Note that the generating function of $\pa(n)$ is
\begin{equation}\label{gen}
\sum_{n=0}^{\infty}\pa(n)z^n= \frac{1}{(1-z^{a_1})\cdots(1-z^{a_r})}.
\end{equation}
Let $D$ be a common multiple of $a_1$, $a_2,\ldots,a_k$. 
Bell \cite{bell} has proved that $\pa(n)$ is a quasi-polynomial of degree $k-1$, with the period $D$, i.e.
\begin{equation}\label{quasi}
\pa(n)=d_{\mathbf a,k-1}(n)n^{k-1}+\cdots+d_{\mathbf a,1}(n)n+d_{\mathbf a,0}(n), 
\end{equation}
where $d_{\mathbf a,m}(n+D)=d_{\mathbf a,m}(n)$ for $0\leq m\leq k-1$ and $n\geq 0$, and $d_{\mathbf a,k-1}(n)$ is
not identically zero.
Sylvester \cite{sylvester},\cite{sylv} decomposed the restricted partition in a sum of ``waves'': 
\begin{equation}\label{wave}
\pa(n)=\sum_{j\geq 1} W_{j}(n,\mathbf a), 
\end{equation}
where the sum is taken over all distinct divisors $j$ of the components of $\mathbf a$ and showed that for each such $j$, 
$W_j(n,\mathbf a)$ is the coefficient of $t^{-1}$ in
$$ \sum_{0 \leq \nu <j,\; \gcd(\nu,j)=1 } \frac{\rho_j^{-\nu n} e^{nt}}{(1-\rho_j^{\nu a_1}e^{-a_1t})\cdots (1-\rho_j^{\nu a_k}e^{-a_kt}) },$$
where $\rho_j=e^{\frac{2\pi i}{j}}$ and $\gcd(0,0)=1$ by convention. Note that $W_{j}(n,\mathbf a)$'s are quasi-polynomials of period $j$.
Also, $W_1(n,\mathbf a)$ is called the \emph{polynomial part} of $\pa(n)$ and it is denoted by $\Pa(n)$.

It is well known that $p(n,k)$, the number of partitions of $n$ with exactly $k$ summands, equals to the number of partitions of $n$ whose largest part is $k$. It follows that
\begin{equation}\label{rest}
p(n,k)=\begin{cases} p_{(1,2,\ldots,k)}(n-k),\; n\geq k \\ 0,\; n<k \end{cases}.
\end{equation}
There is a $1$-to-$1$ correspondence between the partitions of $n$ with $k$ distinct parts and the partitions of $n-\binom{k}{2}$ with
$k$ parts, given by $$a_1<a_2<\ldots<a_k \mapsto a_1\leq a_2-1 \leq \cdots \leq a_k-(k-1).$$ 
Hence
\begin{equation}\label{rest2}
q(n,k)=\begin{cases} p(n-\binom{k}{2},k),\; n\geq k+\binom{k}{2} \\ 0,\; n<k+\binom{k}{2} \end{cases}.
\end{equation}
From \eqref{gen}, \eqref{rest} and \eqref{rest2} it follows that
\begin{align*}
& \sum_{n=0}^{\infty}p(n,k)z^n = \frac{z^k}{(1-z)(1-z^2)\cdots (1-z^k)},\;
 \sum_{n=0}^{\infty}q(n,k)z^n = \frac{z^{k+\binom{k}{2}}}{(1-z)(1-z^2)\cdots (1-z^k)},
\end{align*}
are the generating functions for $p(n,k)$ and $q(n,k)$ respectively.

\newpage
\section{Main results}

Let $D_k$ be the least common multiple of $1,2,\ldots,k$.

\begin{prop}
We have that
$$p(n,k) = f_{k,k-1}(n)n^{k-1}+\cdots+f_{k,1}(n)n+f_{k,0}(n)\text{ for all }n\geq k,$$
where $f_{k,m}(n) = d_{\mathbf k,m}(n-k),\text{ and }\mathbf k=(1,2,\ldots,k)$.
\end{prop}

\begin{proof}
It follows from \eqref{quasi} and \eqref{rest}.
\end{proof}

\begin{teor}\label{formu}
\begin{enumerate}
 \item[(1)] For $n\geq k$ we have that:
  $$ p(n,k) = \frac{1}{(k-1)!} \sum_{\substack{0\leq j_1\leq \frac{D_k}{1}-1,\;%0\leq j_2\leq \frac{D_k}{2}-1
,\ldots,0\leq j_k\leq \frac{D_k}{k}-1 
            \\ j_1+2j_2+\cdots+kj_k\equiv (n-k)\bmod D_k}}\prod_{\ell=1}^{k-1}\left(\frac{n-k-j_1-2j_2-\ldots-kj_k}{D_k}+\ell\right).$$
 \item[(2)] For $n\geq k+\binom{k}{2}$ we have that: \small
$$ q(n,k) = \frac{1}{(k-1)!} \sum_{\substack{0\leq j_1\leq \frac{D_k}{1}-1,\;%0\leq j_2\leq \frac{D_k}{2}-1
,\ldots,0\leq j_k\leq \frac{D_k}{k}-1 
            \\ j_1+2j_2+\cdots+kj_k\equiv (n-k-\binom{k}{2})\bmod D_k}}\prod_{\ell=1}^{k-1}\left(\frac{n-k-\binom{k}{2}-j_1-2j_2-\ldots-kj_k}{D_k}+\ell\right).$$
						\normalsize
\end{enumerate}
\end{teor}

\begin{proof}
(1) The result follows from \cite[Corollary 2.10]{lucrare} and \eqref{rest}.

(2) It follows from (1) and \eqref{rest2}.
\end{proof}

The \emph{unsigned Stirling numbers} $\stir{n}{k}$ are defined by the identity
$$(x)^n:=x(x+1)\cdots(x+n-1)=\sum_{k=0}^n\stir{n}{k}x^k.$$
The \emph{Bernoulli numbers} $B_{\ell}$'s are defined by the identity
$$\frac{t}{e^t-1}=\sum_{\ell=0}^{\infty}\frac{t^{\ell}}{\ell !}B_{\ell}.$$
$B_0=1$, $B_1 = -\frac{1}{2}$, $B_2=\frac{1}{6}$, $B_4=-\frac{1}{30}$ and $B_n=0$ is $n$ is odd and greater than $1$.

\begin{prop}\label{trei}
\begin{enumerate}
\item[(1)] For $n\geq 3$, we have that:
\begin{align*}
& p(n,3)=\sum_{m=1}^3 \frac{(-1)^{m-1}}{6(m-1)!} \sum_{i_1+i_2+i_3=2-m}\frac{B_{i_1}B_{i_2}B_{i_3}}{i_1!i_2!i_3!}2^{i_2}3^{i_3}(n-3)^{m-1} + \\
 & + \frac{1}{12} \sum_{j=2}^3 \sum_{\ell=1}^j \rho_j^{\ell} \sum_{k=0}^2 \frac{1}{6^k} \stir{3}{k+1} 
 \sum_{\substack{ 0\leq j_1\leq 5,\;0\leq j_2\leq 2,\;0\leq j_3\leq 1 \\ j_1+2j_2+3j_3 \equiv \ell (\bmod\;j) }} (j_1+2j_2+3j_3)^k,
\end{align*}
where $\rho_j=e^{\frac{2\pi i}{j}}$.
\item[(2)] For $n\geq 6$, we have that:
\begin{align*}
& q(n,3)=\sum_{m=1}^3 \frac{(-1)^{m-1}}{6(m-1)!} \sum_{i_1+i_2+i_3=2-m}\frac{B_{i_1}B_{i_2}B_{i_3}}{i_1!i_2!i_3!}2^{i_2}3^{i_3}(n-6)^{m-1} + \\
 & + \frac{1}{12} \sum_{j=2}^3 \sum_{\ell=1}^j \rho_j^{\ell} \sum_{k=0}^2 \frac{1}{6^k} \stir{3}{k+1} 
 \sum_{\substack{ 0\leq j_1\leq 5,\;0\leq j_2\leq 2,\;0\leq j_3\leq 1 \\ j_1+2j_2+3j_3 \equiv \ell (\bmod\;j) }} (j_1+2j_2+3j_3)^k,
\end{align*}
\end{enumerate}
\end{prop}

\begin{proof}
(1) Since $1,2,3$ are coprime, the conclusion follows from \cite[Proposition 4.3]{remarks} 
and the fact that $p(n,3)=p_{(1,2,3)}(n-3)$ for $n\geq 3$.

(2) Follows from (1) and the fact that $q(n,3)=p(n-3,3)$ for $n\geq 6$.
\end{proof}

The \emph{Bernoulli polynomials} are defined by
$$B_n(x)=\sum_{k=0}^n\binom{n}{k}B_{n-k}x^k.$$
For $\mathbf a=(a_1,\ldots,a_k)$, the Bernoulli-Barnes numbers (see \cite{barnes}) are 
$$B_j(\mathbf a)=\sum_{i_1+\cdots+i_j=j} \binom{j}{i_1,\ldots,i_k}B_{i_1}\cdots B_{i_k}a_1^{i_1}\cdots a_k^{i_k}.$$
 % Since $p(n,k)=p_{(1,2,\ldots,k)}(n-k)$ and $q(n,k)=p(n-\binom{k}{2},n)$,
% the same assertion is true for $p(n,k)$ and $q(n,k)$.
We consider the determinant:
$$\Delta(k):=\begin{vmatrix} 
\frac{B_1(\frac{1}{D_k})}{1} & \cdots & \frac{B_1(1)}{1} & \cdots & \frac{B_k(\frac{1}{D_k})}{k} & \cdots & \frac{B_k(1)}{k} \\
\frac{B_2(\frac{1}{D_k})}{2} & \cdots & \frac{B_1(1)}{1} & \cdots & \frac{B_{k+1}(\frac{1}{D_k})}{k+1} & \cdots & \frac{B_{k+1}(1)}{k+1}  \\
\vdots & \vdots & \vdots & \vdots & \vdots & \vdots & \vdots \\
\frac{B_{kD_k}(\frac{1}{D_k})}{kD_k} & \cdots & \frac{B_{kD_k}(1)}{kD_k} & \cdots & \frac{B_{kD_k+k-1}(\frac{1}{D_k})}{kD_k+k-1} & \cdots & 
\frac{B_{kD_k+k-1}(1)}{kD_k+k-1} \end{vmatrix}.$$

\begin{teor}\label{dete}
If $\Delta(k)\neq 0$, then $p(n,k)$ can be computed in terms of $B_j(\frac{v}{D_k})$, $1\leq v\leq k$, $1\leq j\leq kD_k$ and 
$B_j(\mathbf k)$, $0\leq j\leq kD_k$, where $\mathbf k=(1,2,\ldots, k)$.
\end{teor}

\begin{proof}
According to \cite[Formula (1.8)]{medi}, we have that:
\begin{equation}\label{eq}
\sum_{m=0}^{k-1}\sum_{v=1}^{D_k}d_{\mathbf k,m}(v)D^{n+m}\frac{B_{n+m+1}(\frac{v}{D})}{n+m+1} = 
\frac{(-1)^{n-1}k!}{(n+k)!}B_{n+k}(\mathbf k)-\delta_{0n},\;(\forall)n\geq 0,
\end{equation}
where $\delta_{0n}$ is the \emph{Kronecker's symbol}. Giving values $0\leq n\leq kD_k-1$ in \eqref{eq},
and seeing $d_{\mathbf k,m}(v)$'s as variables, we obtain a system of $kD_k$ linear equations, with the determinant equal
to $\pm D_k^N\Delta(k)$ for some integer $N\geq 1$. By hypothesis, $\Delta(k)\neq 0$, hence we can determine $d_{\mathbf k,m}(v)$ by solving the system.
From \eqref{quasi}, one has $$p_{\mathbf k}(n)=d_{\mathbf k,k-1}(n)n^{k-1}+\cdots+d_{\mathbf k,1}(n)n+d_{\mathbf a,0}(n).$$
Hence, the conclusion follows from \eqref{rest} and \eqref{rest2}.
\end{proof}

\section{The polynomial part of $p(n,k)$ and $q(n,k)$}

We recall the following basic facts on quasi-polynomials \cite[Proposition 4.4.1]{stanley}:

\begin{prop}\label{cvasi}%(\cite[Proposition 4.4.1]{stanley})
The following conditions on a function $f:\mathbb N \rightarrow \mathbb C$ and integer $D>0$ are equivalent.
\begin{enumerate}
\item[(i)] $f(n)$ is a quasi-polynomial of period $D$.
\item[(ii)] $\sum_{n=0}^{\infty}f(n)z^n = \frac{L(z)}{M(z)}$, where $L(z),M(z)\in \mathbb C[z]$, every zero $\lambda$ of $M(z)$ satisfies $\lambda^D=1$
(provided $\frac{L(z)}{M(z)}$ has been reduced to lowest terms), and $\deg L(z)<\deg M(z)$.
\item[(iii)] For all $n\geq 0$, $f(n)=\sum_{\lambda^D=1} F_{\lambda}(n) \lambda^{-n}$, where each $F_{\lambda}(n)$ is a polynomial function.
Moreover, $\deg F_{\lambda}(n) \leq m(\lambda)-1$, where $m(\lambda)=$ multiplicity of $\lambda$ as a root of $M(z)$.
\end{enumerate}
\end{prop}

We define the \emph{polynomial part} of $f(n)$ to be the polynomial function $F(n)=F_1(n)$, with the notation of Proposition \ref{cvasi}.
The polynomial part $F(n)$ of a quasi-polynomial $f(n)$ gives a rough approximation of $f(n)$, which is useful for studying the asymptotic behaviour of $f(n)$, 
when $n\gg 0$. If $\mathbf a=(a_1,\ldots,a_k)$ is a sequence of positive integers and $\pa(n)$ is the restricted partition function associated to $\mathbf a$,
we denote $\Pa(n)$, the polynomial part of $\pa(n)$. Several formulas of $\Pa(n)$ were proved in \cite{beck}, \cite{dil} and \cite{lucrare}. 

We consider the following functions:
\begin{align*}
 & P(n,k)=P_{(1,2,\ldots,k)}(n-k),\;n\geq k, \\
 & Q(n,k)=P_{(1,2,\ldots,k)}\left(n-k-\binom{k}{2}\right),\;n\geq k+\binom{k}{2},
\end{align*}
and we called them, the \emph{polynomial part} of $p(n,k)$ and $q(n,k)$, respectively.

\begin{teor}\label{poli}
For $n\geq k$, we have that: \small
 \begin{align*}
  & (1)\; P(n,k)=\frac{1}{D_k(k-1)!} \sum_{0\leq j_1\leq \frac{D_k}{1}-1,\;%0\leq j_2\leq \frac{D_k}{2}-1
,\ldots,0\leq j_k\leq \frac{D_k}{k}-1}\prod_{\ell=1}^{k-1}\left(\frac{n-k-j_1-2j_2-\ldots-kj_k}{D_k}+\ell\right).\\
  & (2)\; P(n,k)=\frac{1}{k!} \sum_{u=0}^{k-1} \frac{(-1)^u}{(k-1-u)!}\sum_{i_1+\cdots+i_k=u} \frac{B_{i_1}\cdots B_{i_k}}{i_1!\cdots i_k!}1^{i_1}\cdots k^{i_k}(n-k)^{k-1-u}.
 \end{align*} \normalsize
\end{teor}

\begin{proof}
(1) It follows from \cite[Corollary 3.6]{lucrare} and \eqref{rest}. See also \cite[Theorem 1.1]{dil}.

(2) It tollows from \cite[Corollary 3.11]{lucrare} and \eqref{rest}. See also \cite[p.2]{beck}.
\end{proof}

\begin{cor}\label{cori}
 For $n\geq k+\binom{k}{2}$, we have that: \small
 \begin{align*}
  & (1)\; Q(n,k)=\frac{1}{D_k(k-1)!} \sum_{0\leq j_1\leq \frac{D_k}{1}-1,\;%0\leq j_2\leq \frac{D_k}{2}-1
,\ldots,0\leq j_k\leq \frac{D_k}{k}-1}\prod_{\ell=1}^{k-1}\left(\frac{n-k-\binom{k}{2}-j_1-2j_2-\ldots-kj_k}{D_k}+\ell\right).\\
  & (2)\; Q(n,k)=\frac{1}{k!} \sum_{u=0}^{k-1} \frac{(-1)^u}{(k-1-u)!}\sum_{i_1+\cdots+i_k=u} \frac{B_{i_1}\cdots B_{i_k}}{i_1!\cdots i_k!}1^{i_1}\cdots k^{i_k}\left(n-k-\binom{k}{2}\right)^{k-1-u}.
 \end{align*} \normalsize
\end{cor}

\begin{proof}
 It follows from Theorem \ref{poli} and \eqref{rest2}.
\end{proof}

\section{The Sylvester waves of $p(n,k)$ and $q(n,k)$}

Let $\mathbf k:=(1,2,\ldots,k)$. According to \eqref{wave}, the restricted partition function $p_{\mathbf k}(n)$ can be writen as a sum of "waves",
$p_{\mathbf k}(n)=\sum_{j=1}^k W_j(n,\mathbf k)$.
% where $W_j(n,\mathbf k)$ is the coefficient of $t^{-1}$ in
%$$ \sum_{0 \leq \nu <j,\; \gcd(\nu,j)=1 } \frac{\rho_j^{-\nu n} e^{nt}}{(1-\rho_j^{\nu}e^{-t})(1-\rho_j^{2\nu}e^{-2t})\cdots (1-\rho_j^{k\nu}e^{-kt}) },$$
% where $\rho_j=e^{\frac{2\pi i}{j}}$ and $\gcd(0,0)=1$ by convention. 
We define the functions
\begin{align}\label{wav}
& W_j(n,k)=W_j(n-k,\mathbf k),\;n\geq k,\\ 
& \widetilde W_j(n,k):=W_j(n-k-\binom{k}{2},\mathbf k),\;n\geq k+\binom{k}{2},
\end{align}
and we call them the "waves" of $p(n,k)$ and $q(n,k)$, respectively. 

\begin{obs}\rm
Note that $P(n,k)=W_1(n,k)$ and $Q(n,k)=\widetilde W_1(n,k)$ are
the polynomial parts of $p(n,k)$ and $q(n,k)$.
\end{obs}

\begin{prop}\label{unde}
\begin{enumerate}
 \item[(1)] For any positive integers $1\leq j\leq k \leq n$, we have that:
\begin{align*}
& W_{j}(n,k) = \frac{1}{D_k(k-1)!} \sum_{m=1}^k \sum_{\ell=1}^{j} \rho_j^{\ell} \sum_{t=m-1}^{k-1}  \stir{k}{t+1} (-1)^{t-m+1} \binom{t}{m-1} \cdot \\
& \cdot \sum_{\substack{0\leq j_1\leq D_k-1,\ldots, 0\leq j_k\leq \frac{D_k}{k}-1 \\ j_1+\cdots+ kj_k \equiv \ell (\bmod j)}} D_k^{-t} (j_1+\cdots+kj_k)^{t-m+1} (n-k)^{m-1}.
\end{align*}
\item[(2)] For any positive integers $1\leq j\leq k$ and $n\geq k+\binom{k}{2}$, we have that:
\begin{align*}
& \widetilde W_{j}(n,k) = \frac{1}{D_k(k-1)!} \sum_{m=1}^k \sum_{\ell=1}^{j} \rho_j^{\ell} \sum_{t=m-1}^{k-1}  \stir{k}{t+1} (-1)^{t-m+1} \binom{t}{m-1} \cdot \\
& \cdot \sum_{\substack{0\leq j_1\leq D_k-1,\ldots, 0\leq j_k\leq \frac{D_k}{k}-1 \\ j_1+\cdots+ kj_k \equiv \ell (\bmod j)}} D_k^{-t} (j_1+\cdots+kj_k)^{t-m+1} (n-k-\binom{k}{2})^{m-1}.
\end{align*}
\end{enumerate}
\end{prop}

\begin{proof}
(1) It follows from \cite[Proposition 4.2]{remarks} and \eqref{wav}.

(1) It follows from \cite[Proposition 4.2]{remarks} and (5.2).
\end{proof}

\section{New formulas and congruences for $p(n,k)$ and $q(n,k)$}

We consider the function:
$$ f(n,k):=\# \{(j_1,\ldots,j_k)\;:\; j_1+2j_2+\cdots+kj_k =n,\; 0\leq j_i \leq \frac{D_k}{i}-1,\;1\leq i\leq k \},$$
where $D_k$ is the least common multiple of $1,2,\ldots,k$. Let $d_k:=kD_k - \binom{k+1}{2}$. Note that
$$ f(n,k)=f(d-n,k),\;\text{ for }0\leq n\leq d,\text{ and }f(n,k)=0\text{ for }n\geq d+1.$$
It follows that $F(n,k)=\sum_{\ell=0}^{d_k} f(n,k)x^{\ell}$ is a reciprocal polynomial.
With the above notations we have:

\begin{prop}\label{p61}
 \begin{enumerate}
  \item[(1)] For $n\geq k$ we have that:\small
        $$ p(n,k)=% \sum_{j=0}^{\lfloor \frac{n-k}{D_k} \rfloor} \binom{k+j-1}{j} f(n-k-jD_k) = 
            \sum_{j=\ceil*{\frac{n+\binom{k}{2}}{D_k}} - k}^{\floor*{\frac{n-k}{D_k}}} \binom{k+j-1}{j} f(n-k-jD_k). $$ \normalsize
  \item[(2)] For $n\geq k+\binom{k}{2}$ we have that: \small
        $$ q(n,k)=%\sum_{j=0}^{\lfloor \frac{n-k-\binom{k}{2}}{D_k} \rfloor} \binom{k+j-1}{j} f(n-k-\binom{k}{2}-jD_k) =
                 \sum_{j=\ceil*{\frac{n}{D_k}} -k}^{\floor*{\frac{n-k-\binom{k}{2}}{D_k}}} \binom{k+j-1}{j} f(n-k-\binom{k}{2}-jD_k).$$ \normalsize
 \end{enumerate}
\end{prop}

\begin{proof}
 (1) It follows from \cite[Proposition 2.2]{remarks}, \cite[Corollary 2.3]{remarks} and \eqref{rest}.

 (2) It follows from (1) and \eqref{rest2}.
\end{proof}

\begin{cor}\label{c62}
 \begin{enumerate}
  \item[(1)] For $n\geq k$ we have that:
      $$ (k-1)!p(n,k) \equiv 0\;\bmod\;(j+\ell+1)(j+\ell+2)\cdots(j+k-1),$$
      where $\ell= \floor*{\frac{n-k}{D_k}} - \ceil*{\frac{n+\binom{k}{2}}{D_k}} + k$.
  \item[(2)] For $n\geq k+\binom{k}{2}$ we have that:
$$ (k-1)!p(n,k) \equiv 0\;\bmod\;(j+\ell'+1)(j+\ell'+2)\cdots(j+k-1),$$
      where $\ell'= \floor*{\frac{n-k-\binom{k}{2}}{D_k}} - \ceil*{\frac{n}{D_k}} + k$.
 \end{enumerate}
\end{cor}

\begin{teor}\label{low}
 \begin{enumerate}
  \item[(1)] The inequality $$\lim_{N\to\infty} \frac{\# \{n\leq N\;:\;p(n,k)\equiv 1(\bmod\; 2)\}}{N}\leq \frac{2}{3},$$
        holds for infinitely many positive integers $k$. Moreover, if the above inequality is not satisfied for some positive integers $k$,
        then it holds for $k+1$.
  \item[(2)] Let $m>1$ be a positive integer. For each positive integer $k$ we have 
             $$\lim_{N\to\infty} \frac{\# \{n\leq N\;:\;p(n,k)\not\equiv 0(\bmod\; m)\}}{N} \geq \frac{1}{\binom{k+1}{2}}.$$
 \end{enumerate}
The above results hold if we replace $p(n,k)$ with $q(n,k)$.
\end{teor}

\begin{proof}
(1) Since $p(n,k)=p_{\mathbf k}(n-k)$ for $n\geq k$, where $\mathbf k = (1,2,\ldots,k)$, the result follows from \cite[Theorem 4.2]{graj}.
    (The asymtpotic behaviour is not changed if we replaced $p_{\mathbf k}(n)$ with $p(n,k)$ or $q(n,k)$.)

(2) As above, the result follows from \cite[Theorem 5.2]{graj}.
\end{proof}

{}

\vspace{2mm} \noindent {\footnotesize
\begin{minipage}[b]{15cm}
Mircea Cimpoea\c s, University Politehnica of Bucharest, Faculty of Applied Sciences,\\ 
Department  of  Mathematical Methods  and  Models,  Bucharest, 060042, Romania\\ 
and Simion Stoilow Institute of Mathematics, Research unit 5, P.O.Box 1-764,\\
Bucharest 014700, Romania, E-mail: mircea.cimpoeas@upb.ro, mircea.cimpoeas@imar.ro
\end{minipage}}

\end{document}